\documentclass{amsart}

\newtheorem{theorem}{Theorem}[section]

\theoremstyle{definition}
\newtheorem{definition}[theorem]{Definition}
\newtheorem{example}[theorem]{Example}

\theoremstyle{remark}
\newtheorem{remark}[theorem]{Remark}

\numberwithin{equation}{section}

\begin{document}

\title{CR manifolds, k-contact manifolds, and generalized Sasakian structures}

\author{Janet Talvacchia}
\address{Department of Mathematics and Statistics \\500 College Ave\\ Swarthmore, PA 19081}
\curraddr{}
\email{jtalvac1@swarthmore.edu}
\thanks{}



\keywords{Generalized Geometry, Generalized Sasakian Structures, Generalized Contact Structures, CR Manifolds}

\date{August 24, 2024}


\begin{abstract}In {\cite {[T1]}}, a notion of a generalized Sasakian structure was introduced in the context of generalized contact geometry, the odd dimensional analogue of generalized complex geometry introduced by Hitchin and Gualtieri. We show that k-contact manifolds are generalized Sasakian if and only if they are classically Sasakian.  We show also that strictly pseudo-convex  CR manifolds are always generalized Sasakian.
\end{abstract}
\maketitle

\section{Introduction}\label{S:intro} 
Classically, we know that coK\"ahler  and Sasakian manifolds yield K\"ahler structures on the product manifold $M \times {\mathbb R}$ albeit via different constructions.  It is natural to ask what the analogs of these spaces would be in the generalized contact  geometry setting, the odd dimensional analog of the generalized complex geometry of Hitchin and Gualtieri. (See {\cite {[S], [AG], [PW], [H], [G1], [G2], [G3]} for background on generalized contact and generalized complex structures.) The notion of a generalized co-K\"ahler structure was introduced in \cite{[GT2]} and it was shown that the product of two generalized contact manifolds was generalized K\"aher if each of the generalized contact manifolds was in fact generalized co-K\"ahler.  The two commuting generalized complex structures in this case are constructed using a product construction.  More specifically, if $M_1$ and $M_2$ admit strong generalized contact metric structures $(\Phi_1, E_{+,1}, E_{-, 1}, G_1)$ and $(\Phi_2, E_{+,2}, E_{-,2}, G_2)$ with $[E_{\pm , i} , E_{\mp , i} ] = 0$ and such that $(G_1\Phi_1, G_1E_{+,1}, G_1E_{-, 1})$ and $(G_2\Phi_2, G_2E_{+,2}, G_2E_{-,2})$  are strong as well,  then the generalized complex structures $J_1 = \Phi_1 \times \Phi_2$ and $J_2 = G_1\Phi_1 \times G_2\Phi_2$ yield a generalized K\"ahler structure on $M_1 \times M_2$. (See  \cite {[GT2]}).  Sasakian, k-contact and CR manifolds fall outside of this categorization. In all these cases, even though one can construct a  strong generalized contact metric structure $(\Phi , E_{+}, E_{-}, G)$, the  corresponding structure $(G\Phi, GE_{+}, GE_{-,}, G)$ is never strong. In \cite {[T]}, it was shown that a notion of generalized Sasakian could not arise from a simple product construction.

A notion of generalized Sasakian structure was introduced in \cite{[T1]} such that if $M$ was generalized Sasakian, $M \times {\mathbb R}$ was generalized K\"ahler.  The commuting complex structures in this case are formed by first constructing a generalized complex structure $J_1 = \Phi_1 \times \Phi_2$ from strong generalized contact structures on $M$ and ${\mathbb R}$ while the second generalized complex structure $J_2$ is obtained by making a change of gauge of the underlying Poisson structure corresponding to $J_1$.  Classical Sasakian manifolds were shown to be generalized Sasakian and the set of generalized Sasakian and the set of generalized coK\"ahler spaces were shown to have no overlap. Thus there are two distinct ways to possibly get generalized K\"ahler structures on the product of two strong generalized contact manifolds.

In this paper, we consider where k-contact and CR manifolds fall in this landscape.  The two main theorems are the following:
\begin {theorem} A k-contact manifold is generalized Sasakian if and only if it is classically Sasakian.
\end{theorem}

\begin{theorem}
Any strictly pseudo-convex CR manifold is generalized Sasakian.
\end{theorem}

\section{Contact Metric Structures, k-contact structures and CR structures}
We begin with a review of k-contact structures and CR structures. We follow the classic texts due to Boyer and Galicki \cite{[BG]} and Blair \cite{[B]} in this exposition. Throughout this paper we let $M$ be a smooth manifold of dimension $2n+1$.

\begin{definition}  $M$ is a contact manifold if there exists a 1-from $\eta$ such that $\eta \wedge (d\eta)^n \neq 0$.  A contact  structure on $M$ is an equivalence class of such 1-forms where $\eta^\prime \sim  \eta$ if there exist a nowhere vanishing function $f$ on $M$ such that $\eta^\prime = f\eta$. 
\end{definition}

A contact structure gives rise to a co-dimension one sub-bundle of $TM$, $D= {\rm ker}\   \eta$ and there exists a unique vector field $\xi $ such that $\eta(\xi) = 1$ called the Reeb vector field.

\begin{definition}An almost contact structure on $M$ is a triple $(\phi, \xi, \eta)$ where $\phi$ is a $(1,1)$ tensor, $\xi$ is a vector field, and $\eta$ is a 1-form satisfying $\phi^2 =  -{\rm Id} + \xi \otimes \eta$, $\eta(\xi ) = 1$, and  $\phi(\xi ) = 0$.  
\end{definition}

The vector field $\xi$ defines the characteristic foliation $ F_\xi$ with one dimensional leaves and the kernel of $\eta$ defines the co-dimension one sub-bundle $D = {\rm ker} \  \eta$.  If the one form $\eta$ satisfies $\eta \wedge (d\eta)^n \neq 0$ then it defines a contact structure on $M$ and $\xi$ is the Reeb vector field.  Since $\phi |_D$ satisfies $\phi ]_D^2 = -{\rm Id}$, $D$ decomposes as a direct sum of the $\pm \sqrt {-1}$ eigenbundles, $D = D^{(1,0)} \oplus D^{(0,1)}$.

\begin{definition}Given the almost contact triple $(\phi, \xi, \eta)$,  a metric $g$ is compatible with the almost contact structure if
$g(\phi(X), \phi(Y) )= g(X,Y) - \eta (X) \eta (Y)$ for all section $X, Y$ of $TM$.. 
\end{definition}
 Such metrics always exist and then we say the quadruple $(\phi, \xi, \eta , g)$ defines an almost contact metric structure on $M$.  If $\eta $ defines a contact structure, we call it a contact metric structure.

\begin{definition}A contact metric structure is called k-contact if $L_\xi \phi = 0$.
\end{definition}

\begin{definition}
A contact metric structure is called normal if the Nienhuis torsion tensor of $\phi$, $N_\phi(X,Y) = [\phi X, \phi Y] + \phi^2 [X,Y] -[\phi X, Y]  - \phi [ \phi X, Y]$ where $X, Y$ are section of $TM$, satisfies $N_\phi = -2\xi \otimes d\eta$.
\end{definition}

\begin{definition}
A normal contact metric manifold is called Sasakian.
\end{definition}

We explore the difference between k-contact and Sasakian structures with more detail in a way that will be useful in what follows.
Given a contact metric structure $(\phi, \xi, \eta , g)$, the condition $N_\phi (X, Y) = -2\xi \otimes d\eta (X,Y)$ for vector fields $X$ and $Y$ on $TM$ reduces to

 $$ (*)\  \  \  \  \  -[X, Y] + X\eta (Y) \xi - Y\eta (X) \xi  = 0$$
Consider a local frame $\lbrack \xi , Z_i , {\bar Z}_i \rbrack$ of $TM \otimes {\mathbb C}$ where $Z_i \in D^{(1,0)}$.  First let $X = \xi$ and $Y = Z_i$.  Then (*) becomes
$$-[\xi , Z_i] - \phi([\xi , \phi(Z_i)]) = 0  .$$

Thus,
$$ \phi([\xi , {\sqrt -1}Z_i)] = -[\xi , Z_i] $$

which implies
$$\phi([\xi , Z_i])] = {\sqrt -1}[\xi , Z_i] .$$
That is, $[\xi , Z_i] \in D^{(1,0)}$.

Observe that (*) is also satisfied if $X  = Z_i$ and $Y= {\bar Z}_i$. (One computes 0=0.)

If we let $X = Z_i$ and $Y = Z_j$, we get that 
$$\phi([Z_i, Z_j]) = {\sqrt -1}[Z_i , Z_j] .$$
That is, $[Z_i, Z_j] \in D^{(1,0)}$.

So the condition $M$ is Sasakian is equivalent to the two conditions  $[\xi , D^{(1,0)}] \subset D^{(1, 0)}$ and $[  D^{(1,0)} ,  D^{(1,0)}] \subset D^{(1,0)}$.  The condition that $M$ is k-contact is equivalent to just the first condition holding.  We see this by expanding the defining relation $L_\xi \phi = 0$ using the definition of the Lie derivative of a tensor. Specifically, if $X \in D^{(1, 0)}$, then $0 = L_\xi \phi (X) = {\sqrt -1}[\xi , X] - \phi([\xi , X])$.

\begin{definition}
 Let $T^{\mathbb C}M = TM\otimes {\mathbb C}$ be the complexified tangent bundle of $M$.  Let $H$ be a $C^\infty$ complex sub-bundle of $T^{\mathbb C}M$ of dimension $l$.  A CR structure is a pair $(M, H)$ such that $H_p \cap {\bar H}_p = \lbrace  0 \rbrace$ and $H$ is involutive.
 \end{definition}

Given  a CR structure $(M, H)$, there exists a unique sub-bundle $D$ of $TM$ such that $D^{\mathbb C} = H \oplus {\bar H}$ and a unique bundle map $J:D \rightarrow D$ such that $J^2 = -{\rm Id}$
and $D^{(1,0)} = H$.  
Consider the case now where $M$ has real dimension 2n+1 and  $H$ has  complex dimension $n$. Consider the space $N_x$ of all 1-forms $\alpha$ such that $D\subset {\rm ker}\   \alpha$.  This defines a real line bundle $N \subset T^*M$.  If $M$ is orientable,  then $N$ admits a nowhere vanishing section $\eta$.  The Levi form is defined by
$$L_\eta (X, Y) = d\eta (X, JY) \  \  \  \   X, Y \in D$$
If $L_\eta$ is non-degenerate, $\eta \wedge (d\eta )^n \neq 0 $ and $\eta $ defines a contact form.  If $L_\eta$ is positive definite as well, we say that the CR structure $(M, H)$ is strictly pseudo-convex.  In this case, using the direct sum decomposition $TM = D \oplus \lbrace \xi \rbrace$, we can extend $L_\eta$ to a metric $g$ on $M$ by setting  $g(\xi ,\xi ) =  1$, $g(\xi ,X) = 0$ for  $X \in D$, and $ g(X, Y)  = L_\eta(X,Y)$ for all $X, Y \in D$.  Also, we can extend $J$ to a tensor $\phi $ on $M$ by $\phi ( \xi ) = 0$ and $\phi (X) = JX \  {\rm for } \  X \in D$.  Thus we see that a strictly pseudo-convex CR manifold carries a contact metric structure.  

Lastly,  we make a note regarding $d\eta$ and the existence of Poisson structures on a strictly pseudo-convex CR manifold $M$.  If $d\eta $ is non-degenerate, then viewing $d\eta$ as a map from $TM$ to $T^*M$, we see that $(d\eta )^{-1}$ is a bivector field .  Since $d^2\eta =  0$,  $(d\eta )^{-1}$ defines a Poisson structure $\pi$. (See \cite {[CFM]} , page 32, Prop 2.18).

\section{k-contact and CR structures as generalized contact structures}\label{B:Back}

\indent    We use the definition of a generalized contact structure given by Sekiya  (see \cite{[S]}). Recall that we have set  $M$ to be a smooth manifold of dimension $2n+1$. Consider the big tangent bundle, $TM\oplus ~T^*M$.  We define a neutral metric on $TM\oplus T^*M$ by$$  \langle X + \alpha  , Y + \beta \rangle =  \frac{1}{2} (\beta (X) + \alpha (Y) )$$ and the Courant bracket by $$[[X+\alpha, Y+ \beta ]] = [X,Y] + {\mathcal L}_X\beta -{\mathcal L}_Y\alpha -\frac{1}{2} d(\iota_X\beta - \iota_Y\alpha)$$ where $X, Y \in TM$ and $\alpha ,\beta  \in T^*M$ . A sub-bundle of $TM\oplus T^*M$ is said to be involutive  or integrable if its sections are closed under the Courant bracket.

\begin{definition} \cite{[S]} A generalized almost contact structure on $M$ is a triple $(\Phi, E_\pm)$ where $\Phi $ is an endomorphism of $TM\oplus T^*M$, and $E_+$ and $E_-$ are sections of $TM\oplus T^*M$ which satisfy
\begin{equation}
\Phi + \Phi^{*}=0
\end{equation}
\begin{equation}\label{phi}
\Phi \circ \Phi = -Id + E_+ \otimes E_- + E_- \otimes E_+
\end{equation}
\begin{equation}\label{sections}
 \langle E_\pm, E_\pm \rangle = 0,  \  \    2\langle E_+, E_-  \rangle = 1.
\end{equation}

\end{definition}
Now, since $\Phi$ satisfies $\Phi^3 + \Phi =0$, we see that $\Phi$ has $0$ as well as $\pm \sqrt{-1}$ eigenvalues when viewed as an endomorphism of the complexified big tangent bundle $(TM\oplus ~T^*M) ~\otimes { ~\mathbb C}$.  The kernel of $\Phi$ is $L_{E_+} \oplus L_{E_-}$ where $L_{E_\pm}$ is the line bundle spanned by ${E_\pm}$.  Let $E^{(1,0)}$ be the $\sqrt{-1}$ eigenbundle of $\Phi$.  Let $E^{(0,1)}$ be the $-\sqrt{-1}$ eigenbundle. Observe:

$$
E^{(1,0)} = \lbrace X + \alpha - \sqrt{-1}  \Phi (X + \alpha ) |  \langle E_\pm, X + \alpha \rangle = 0 \rbrace
$$

$$
E^{(0,1)} = \lbrace X + \alpha + \sqrt{-1}  \Phi (X + \alpha ) | \langle E_\pm, X + \alpha \rangle = 0 \rbrace .$$

Then the complex vector bundles
$$L^+ = L_{E_+} \oplus E^{(1,0)}$$
and
$$L^- = L_{E_-} \oplus E^{(1,0)}$$
are maximal isotropics.
\begin{definition} \cite{[PW]}
A generalized almost contact structure $(\Phi,E_{\pm})$ is a  generalized contact structure if either $L^{+}$ or $L^{-}$ is closed with respect to the Courant bracket. The generalized
contact structure is strong if both $L^{+}$ and $L^{-}$ are closed with respect to the Courant bracket.
\end{definition}

\begin{definition} \cite{[GT2]}
A  generalized almost contact structure $(M, \Phi,E_{\pm})$ is a normal generalized contact structure if  $\Phi$ is strong and $[[E_+, E_-] ] = 0$.
\end{definition}

\begin{remark}  This definition of normality is motivated by Theorem 1 of \cite{[GT1]} that shows that product of two generalized almost contact spaces $(M_1, \Phi_1, E_{\pm 1})$  and $(M_2, \Phi_2, E_{\pm 2})$ induces a standard generalized almost  complex structure on $M_1 \times M_2 $.  The generalized complex structure is integrable if each $\Phi_i$ is strong and  $[[E_{+i},  E_{-i}]] = 0$. 
\end{remark}

Here are the standard examples:

\begin{example}\label{almost contact example}  \cite{[PW]}
Let $(\phi , \xi, \eta)$ be a normal almost contact structure on a manifold $M^{2n+1}$.  Then we get a generalized almost contact structure by setting
$$ \Phi = \left ( \begin{array}{cc}  \phi & 0 \\ 0 & -\phi^* \end{array} \right ),\  \   E_+ = \xi,\  \   E_-= \eta $$  where $(\phi^*\alpha )(X) = \alpha (\phi (X)), \   X \in TM,\   \alpha \in T^{*}M$. Moreover, $(\Phi, E_\pm)$ is an example of a strong generalized almost contact structure.
\end{example}

\begin{example} \label{contact example} \cite{[PW]}
Let $( M^{2n+1}, \eta )$ be a contact manifold with $\xi $ the corresponding Reeb vector field so that
$$ \iota_\xi d\eta = 0 \  \  \  \eta ( \xi ) = 1.$$
Then $$\rho ( X) := \iota_X d\eta - \eta ( X)\eta$$ is an isomorphism from the tangent bundle to the cotangent bundle.  Define a bivector field by
$$\pi (\alpha , \beta ) := d\eta (\rho^{-1}( \alpha ), \rho^{-1}( \beta )),$$
where $\alpha, \beta \in T^{*}. $
We obtain a generalized almost contact structure by setting
$$ \Phi = \left ( \begin{array}{cc}  0 & \pi \\ d\eta & 0 \end{array} \right ),\  \   E_+ = \eta,\  \   E_-= \xi .$$
In fact, $(\Phi, E_{\pm})$ is an example which is not strong.
\end{example}

\begin{definition}\cite{[G1]}
A  generalized metric $G$ on $M$ is an automorphism of $TM\oplus T^*M$ such that $G^{*}=G$ and $G^{2}=1.$ 
\end{definition}

\begin{definition} \cite{[S]}
A generalized almost contact metric structure is a generalized almost contact structure
$(\Phi , E_{\pm})$ along with a generalized  metric $G$ that satisfies
\begin{equation}\label{compatcond}
-\Phi G \Phi = G - E_+ \otimes E_+ -E_- \otimes E_-.
\end{equation}
\end{definition}

\begin{definition}\cite{[G1]}
Let $B$ be a closed two-form which we view as a map from $T \rightarrow T^{*}$ given by interior product. Then the invertible bundle map
$$e^{B}:= \left ( \begin{array}{cc}  1 & 0 \\ B & 1 \end{array} \right):X+\xi \longmapsto X+\xi + \iota_{X}B$$
is called a B-field transformation.
\end{definition}
In \cite {[GT2]} it was proved that a $B$-field transformation of a normal generalized contact metric structure,
 $(e^B\Phi e^{-B},e^{B}E_{\pm},e^{B}G e^{-B})$, is again a normal contact structure.

\begin{example}
If $M$ is either k-contact or  strictly pseudo-convex CR, we can form the generalized contact metric structure associated to its almost contact metric structure $(\phi , \xi , \eta , g)$ by defining

$$ \Phi = \left ( \begin{array}{cc}  \phi & 0 \\ 0 & -\phi^* \end{array} \right ),\  \   E_+ = \xi,\  \   E_-= \eta , \  \  \   G= \left ( \begin{array}{cc}  0 & g^{-1} \\ g & 0 \end{array} \right )$$  where $(\phi^*\alpha )(X) = \alpha (\phi (X)), \   X \in TM,\   \alpha \in T^{*}M$. 

Note that in both cases $d\eta \neq 0$.  In the case that $M$ is a CR manifold, the involutively of $H = D^{(1,0)}$ implies that the associated generalized contact  structure is normal.
\end{example}

\begin{theorem} \cite {[T1]} A normal generalized contact manifold $(M, \Phi, E_+ , E_-  )$ admits a canonical Poisson structure $\pi_M = \pi_0 + e_+ \wedge e_-$ where $\pi_0$ is the canonical Poisson structure associated to the generalized complex structure $\Phi |_{(E_+ \oplus E_-)^\perp}$ and $e_\pm = pr_{TM}E_\pm$.
\end{theorem}

\begin{definition} \cite{[T1]}
A normal generalized contact metric space $(M, \Phi , E_+ , E_- , G)$ is defined to be generalized Sasakian if $ (I + e^t(d\eta) \pi_M)$ is invertible as a map from $T^*M$ to $T^*M$  for all values of
 $t \in  \mathbb{R}$ where  $\pi_M$ is the canonical Poisson structure on $M$, $\eta = pr_{T^*M}E_+ + pr_{T^*M}E_-$ and $d\eta \neq 0$.
\end{definition}



The following theorem was proved in \cite{ [T1]}

\begin{theorem}
Let  $(M_1, \Phi , E_{+1} , E_{- 1}, G_1)$ be a generalized Sasakian space with canonical Poisson structure $\pi_{M_1}$.  
Let  $(M_2, \Phi , E_{+2} , E_{- 2}, G_2)$ be a generalized Sasakian space with canonical Poisson structure $\pi_{M_2 } = 0$.
Then $M_1 \times M_2$ is generalized K\"ahler.

\end{theorem}

\section{K-contact structures are Generalized Sasakian if and only if they are Sasakian}

\begin{theorem}
A k-contact manifold is generalized Sasakian if and only if it is Sasakian.
\end{theorem}
\begin{proof}
Let $(M, \phi, \eta, \xi, g )$ be a k-contact manifold. If $M$ is generalized Sasakian then it has a normal generalized contact  metric structure associated to it of the form 
$$ (M, \  \  \Phi ,\  \   E_+ = \xi  + \eta_+ ,\  \   E_-  =  \eta_-, \  \  G ).$$ where $\eta = \eta_+ + \eta_-$.
Let $D$ = ker $\eta$ and $D \otimes {\mathbb C} = D^{(1,0)} \oplus D^{(0,1)}$.
Then $E^{(1,0)} = D^{(1,0)} \oplus (D^{(1, 0)})^*$ and $E^{(0,1)} = D^{(0, 1)} \oplus (D^{(0,1)})^*$.  Since $(M, \phi, \eta, \xi, g )$ is k-contact, $[ \xi, D^{(1,0)}] \subset D^{(1,0)}$.  Since $ (M, \Phi, E_+, E_- , G)$ is strong,
$[[ E^{(1,0)}, E^{(1,0)}]] \subset E^{(1,0)}$ which implies $[D^{(1,0)}, D^{(1,0)}] \subset D^{(1,0)}$ as well.  Hence $(M, \phi, \eta, \xi, g )$  is Sasakian.

The reverse direction is straightforward. If $M$ is Sasakian it is both k-contact and generalized Sasakian.

\end{proof}

\section{Strictly Pseudo-convex  CR structures are Generalized Sasakian}
\begin{theorem}
Let M have a strictly pseudo-convex  CR structure.  Then M is a generalized Sasakian manifold.
\end{theorem}

\begin{proof}
Given a strictly pseudo-convex CR structure on $M$, construct the almost contact metric structure $(M, \phi, \xi , \eta, g)$ associated to it.  Note that since the  CR structure is strictly pseudo-convex, $d\eta$ is non-degenerate and $(d\eta)^{-1}$ defines a Poisson structure $\pi$. From the almost contact metric structure, construct the normal generalized contact metric structure structure 
$$ \Phi = \left ( \begin{array}{cc}  \phi & 0 \\ 0 & -\phi^* \end{array} \right ),\  \   E_+ = \xi,\  \   E_-= \eta ,\   \   G = \left ( \begin{array}{cc}  0 & g^{-1} \\ g & 0 \end{array} \right ).$$  The generalized complex structure on $ (E_+  \oplus E_- )^\perp$ obtained by restricting $\Phi$ is a generalized complex structure that arises from a classical complex structure. Hence, the corresponding canonical poisson structure associated to the generalized contact metric structure is the zero Poisson structure. (See \cite {[T1]} for the construction of the generalized complex structure obtained by restricting $\Phi$ and the canonical Poisson structure associated to the generalized contact structure.  See  \cite {[AB]} for details about the canonical Poisson structure associated to a generalized complex structure that arises from a classical complex structure.)  However, the B transform of this generalized contact metric structure by $B = d\eta$, 
$$ \Phi^\prime =exp (B)\Phi  exp (-B ),\  \  E_+^\prime = exp(B) E_+ = \xi +\xi  \lrcorner B ,$$ $$ E_-^\prime = exp (B) E_-= \eta ,\   \  G^\prime =  exp(B) G  exp(-B)  $$
is also a normal generalized contact metric structure.  The B transform shifts the canonical symplectic foliation associated to the original generalized complex structure on $ (E_+  \oplus E_- )^\perp$ (see \cite {[AB]}) so that
 $\pi  +\xi \wedge 0$ is the canonical Poisson structure on $M$ associated with this B transformed generalized contact structure.  Since $(I + e^t(d\eta)\pi) = (1+e^t) I$ is invertible for all $t$, $M$ is generalized Sasakian and, hence, $M \times {\mathbb R}$ admits a generalized K\"ahler structure.
\end{proof}

\begin{remark} Since there are strictly pseudo-convex CR manifolds that are not classically Sasakian, this theorem provides a way of generating examples of generalized Sasakian manifolds that are not classically Sasakian.
\end{remark}


\begin{thebibliography}{10}

\bibitem{[AB]}
M. Abouzaid and R Boyarchenko,
\emph{Local Structure of Generalized Complex Manifolds}, J. Symplectic Geom. {\bf 4} (2006) no. 1, 43-62.


\bibitem{[AG]}M.Aldi and D. Grandini.
\emph{Generalized Contact Geometry and T-duality}, J Geom. Phys. {\bf 92} (2015), 78-93.


\bibitem{[B]} D. Blair
\emph{ Riemannian Geometry of Contact and Symplectic Manifolds}, Boston, Birkhauser, 2002.


\bibitem{[BG]} C. Boyer and K Galicki
\emph{ Sasakian Geometry}, Oxford, Oxford University Press, 2008.




\bibitem{[CFM]} M. Crainic, R. Fernandes, and  I. Marcut
\emph{ Lectures On Poisson Geometry}, Providence, A.M.S., 2021.




\bibitem{[GT1]}
R. Gomez. and J. Talvacchia,
\emph{On Products of Generalized Contact Structures},  Geom. Dedicata {\bf 175} (2015), 211-218.

\bibitem{[GT2]}
R. Gomez. and J. Talvacchia,
\emph{Generalized CoK\"ahler Geometry  and an Application to Generalized K\"ahler Structures}, J Geom.Phys. {\bf 98} (2015), 493-503.


\bibitem{[G1]}
 M. Gualtieri,
 \emph{Generalized  Complex Geometry},
 Ph.D thesis, Oxford Univ., 2003, arXviv:math.DG.0401221.

 \bibitem{[G2]}
 M. Gualtieri,
 \emph{ Generalized Complex Geometry},
 Ann. of Math. {\bf (2) 174} (2011) no. 1, 75-123.

 \bibitem{[G3]}
 M. Gualtieri,
 \emph{ Generalized K\"ahler Geometry},
Commun. Math. Phys. {\bf 331} (2014), 297-331.


 \bibitem{[H]}
N.J.Hitchin,
 \emph{Generalized Calabi-Yau Manifolds}, Q. J. Math. {\bf 54} (2003), 281-308.






 \bibitem{[PW]}
 Y.S. Poon and A. Wade,
 \emph{Generalized Contact Structures}, J. Lond. Math. Soc. {\bf (2) 83} (2011) no.2, 333- 352.


 \bibitem{[S]}K. Sekiya,
 \emph{Generalized Almost Contact Structures and Generalized Sasakian Structures}, Osaka J. Math. {\bf  52}  (2015), 43-59.
 

 
 \bibitem{[T]} J. Talvacchia
 \emph{ Commuting Pairs of Generalized Contact Metric Structures} J. Geom. Symmetry  Phys.,
{\bf 46} (2017), 37-50.

\bibitem{[T1]} J. Talvacchia. 
\emph{Generalized Sasakian Structures from a Poisson Geometry Viewpoint.} Geom Dedicata 217, 87 (2023). https://doi.org/10.1007/s10711-023-00821-y




\end{thebibliography}
\end{document}